\documentclass[12pt,letterpaper]{article}
\usepackage[hebrew,english]{babel}

\usepackage{amsfonts,amssymb}
\usepackage{amsmath}

\numberwithin{equation}{section}

\newtheorem{theorem}{Theorem}

\newtheorem{definition}{Definition}
\newtheorem{remark}{Remark}

\newtheorem{lemma}{Lemma}

\newcommand{\qed}{\(\Box\)}
\newcommand{\N}{\mathbb{N}}
\newcommand{\B}{\mathbb{B}}

\newcommand{\Z}{\mathbb{Z}}
\newcommand{\F}{\mathbb{F}}

\def\deg{\hbox{\rm deg}\,}

\numberwithin{lemma}{section} \numberwithin{theorem}{section}
\numberwithin{prop}{section} \numberwithin{remark}{section}
\numberwithin{definition}{section} \numberwithin{corollary}{section}

\title{Polynomial largeness of sumsets and totally ergodic sets}
\author{Alexander Fish}
\begin{document}

\maketitle
\begin{abstract}
We prove that a sumset of a TE subset of \( \N\) (these sets can be viewed as ``aperiodic" sets) with a set of positive upper density intersects a set of  values of any polynomial with integer coefficients., i.e. for any \( A \subset \N \) a TE set, for any \( p(n) \in \Z[n]: \, \deg{p(n)} > 0, p(n) \to_{n \to \infty} \infty \) and any subset \( B \subset \N \) of positive upper density  we have \( R_p = A+B \cap \{ p(n) \, | \, n \in \N \}  \neq \emptyset\). 
For \( A \) a WM set (subclass of TE sets) we prove that \( R_p \) has lower density 1. In addition we obtain a generalization of the latter result to the case of several polynomials and several WM sets (see theorem \ref{theorem3}). 
\end{abstract}

\section{Introduction}

\noindent We start from the following question: Can we provide non-trivial examples of subsets  \( A \subset \N \) (density of A should be as small as we wish) such that for any \( B \subset \N \) of positive density the set \( A + B  \, \, (A+B = \{ a+b \, | \, a \in A, b \in B \}) \) intersects a set of values of any polynomial with integer coefficients with a positive leading coefficient? It means that \( \forall p(n) \in \Z[n] \) such that \( p(n) \to_{n \to \infty} \infty\) we have \( (A+B) \cap \{p(n) \, | \, n \in \N \} \neq \emptyset \).
We introduce a notion of a \textit{``p-good"} set (``p" stands for polynomials). A set \( A \subset \N \) is a 
 \textbf{p-good} if for every \( B \subset \N \) of positive upper density and every \( p(n) \in \Z[n], \, p(n) \to_{n \to \infty} \infty \) we have \( (A +B) \cap \{ p(n) | n \in \N \} \neq \emptyset \).

\noindent If we fix a polynomial \( p \) of degree greater or equal than \( 2 \) then for infinitely many primes
\( q \in P \) we have that the set of values \( \{p(n) \, | \, n \in \N \}\) projected on \( \mathbb{F}_p \) is not surjective. The latter follows from the fact that for a given polynomial \( p \in \Z[n] \) there are infinitely many primes \( q \) such that \( p(n) \) projected to \( \F_{q}[n] \) is splitting, see \cite{lenstra}. There are two possible cases. In the first case \( p(n) \in \F_{q}[n]\) has at least two different roots. Then it means that zero has at least two pre-images. So, the projection of    \( \{p(n) \, | \, n \in \N \}\)  on \( \F_p \)
is not surjective. In the second case, we have that \( p(n) \) covers  just all roots of degree \( d \), where \( d = \deg{p} \). We know that it can
not be more than \( \frac{q-1}{d} \) such numbers.

\noindent So for a fixed \( p(n) \in \Z[n] \) such that \( \deg{p} \geq 2 \) there are infinitely many primes \( q \)
 such that for every  congruence class \( A \) modulo \( q \)  there exists another congruence class \( B \)
 modulo \( q \)  with \( (A + B) \cap  \{p(n) \, | \, n \in \N \} = \emptyset \).

\noindent So, for a periodic set \( A \) we don't have any hope that for any \( B \subset \N \) of positive density  the set \( A + B \) intersects non-trivially a set of values of every polynomial.

\noindent The natural question is the following. If \( A \) does not
exhibit any periodicity (in dynamical context it is equivalent to
total ergodicity of \( A \)) does it follow that  \( A \) is p-good?
An answer to this question is affirmative. Before stating the
theorem one gives a formal definition of a \textit{TE set} and of \textit{WM set} (we will need this notion later).

\noindent We remind basic notions of ergodic theory: \textbf{measure
preserving system, generic point, ergodicity, total ergodicity and
weak mixing}.

\noindent
Let \( X \) be a compact metric space, \( \mathbb{B} \)  the Borel
\( \sigma\)-algebra on \( X \), \( T:X \rightarrow X \) be a
continuous map and \( \mu \) a probability measure on \(
\mathbb{B} \) such that for every \( B \in
\mathbb{B} \) we have \( \mu(T^{-1}B) = \mu(B) \).

\noindent The quadruple \( (X,\B,\mu,T)\) is called a \textbf{measure preserving system}.

\noindent For a compact metric space \( X \) we denote by \( C(X)
\) the space of continuous functions on \( X \) with the uniform
norm.
\begin{definition} Let \( (X,\mathbb {B},\mu ,T) \) be a measure
preserving system. A point \( \xi \in X \) is called
\textbf{generic} if for any \( f\in C(X) \) we have
\begin{equation}
\label{limit} \lim _{N\rightarrow \infty }\frac{1}{N}\sum
^{N-1}_{n=0}f(T^{n}\xi )=\int _{X}f(x)d\mu (x).
\end{equation}
\end{definition}


\noindent We recall the definitions of ergodic, totally ergodic  and weakly mixing measure preserving systems.
\begin{definition}
A measure preserving system \( (X,\mathbb{B},\mu,T) \) is called
\textbf{ergodic} if any measurable set \( B \in \B \) which is invariant under \( T \), i.e. \( T^{-1}B = B \) has measure \( 0 \) or \(1\).
\newline
A measure preserving system \( (X,\mathbb{B},\mu,T) \) is called
\textbf{totally ergodic} if for every \( n \in \N \) the system \( (X,\mathbb{B},\mu ,T^n) \) is ergodic.
\newline
A measure preserving system \( (X,\mathbb{B},\mu,T) \) is called
\textbf{weakly mixing} if  the system \( (X\times X,\mathbb{B}_{X
\times X},\mu \times \mu ,T\times T) \) is ergodic.

\end{definition}

\noindent
Let \( \xi (n) \) be any \( \{0,1\}- \)valued sequence. There is a
natural dynamical system \( (X_{\xi },T) \) connected to the
sequence \( \xi \):

\noindent On the  compact space \( \Omega
=\{0,1\}^{\mathbb {N}} \)  endowed with the Tychonoff topology,
we define a continuous map \( T:\Omega \longrightarrow \Omega  \)
by  \(
(T\omega )_{n}=\omega _{n+1} \). Now for any \( \xi \) in \(
\Omega \) we define \( X_{\xi } \) to be \(
\overline{(T^{n}\xi )_{n\in \mathbb {N}}} \subset \Omega\). 

\noindent Let \( S \) be a subset of \( \mathbb {N} \). Choose \(
\xi =1_{S} \) and assume that for an appropriate measure \( \mu \),
the point \( \xi \) is generic for \( (X_{\xi },\mathbb {B},\mu
,T) \). We can attach to the set \( S \) dynamical properties
associated with the system \( (X_{\xi },\mathbb {B},\mu ,T) \).
\newline
 \( S \) is called \textit{totally ergodic}  if the measure preserving system \( (X_{\xi },\mathbb
{B},\mu ,T) \) is totally ergodic.

\noindent
 \( S \) is called \textit{weakly mixing}  if the measure preserving system \( (X_{\xi },\mathbb
{B},\mu ,T) \) is weakly mixing.

\noindent We remind the notion of a density of a subset of \( \N\).
\begin{definition}
Let \( S \subset \mathbb{N} \). If the limit of \( \frac{1}{N}
\sum_{n=1}^{N} 1_S(n) \) exists
 as \( N \rightarrow \infty \) we call it the \textbf{density} of \( S \) and denote it by \( d(S) \).
\end{definition}

\begin{remark}
The upper and lower limits of the sequence \( \frac{1}{N}
\sum_{n=1}^{N} 1_S(n) \) always exist and they are called
\textbf{upper} \textnormal{(\( \overline{d}(S)\))} and, correspondingly,
\textbf{lower densities} \textnormal{(\( \underline{d}(S)\))} of \( S \).
\end{remark}

\noindent In our discussion of TE (WM) sets corresponding to totally
ergodic (weakly mixing) systems, we add the condition that the
density of a set (which exists) should be positive. Without making
this assumption any set of zero density would be in our class of
totally ergodic sets (weakly mixing sets). But a set of zero density
might be as bad as we like. Therefore we concerned only with sets of
positive density.

\begin{definition}
\label{def_TE_WM} A subset \( S \subset \mathbb{N} \) is called a
\textbf{TE set} (\textbf{WM set}) if \( S \) is totally ergodic
(weakly mixing) and the density of \( S \) is positive. That is to
say, \( 1_S \) is a generic point of the totally ergodic  (weakly
mixing) system \( (X_{1_S},\mathbb{B},\mu ,T) \) and \( d(S) > 0 \).
\end{definition}

\begin{remark}
Any WM set is a TE set.
\end{remark}

\noindent In the paper we prove that any TE set is p-good:
\begin{theorem}
\label{theorem1} Let \( A \subset \N \) be a TE set. Then for any \(
B \subset \N \) of positive upper density and any non-constant
polynomial \( p(n) \in \Z[n] \) with a positive leading coefficient
we have \( A+B \cap \{p(n) \, | \, n \in \N \} \neq \emptyset \).
Moreover, if density of \( B \) exists and positive then the set \(
R_p  = \{ n \in \N \, | \, p(n) \in A + B \}\) is syndetic (it has
bounded gaps).
\end{theorem}

\noindent If we require from \( A \) to be WM set, then we can prove
that the set \( R_p \) is of lower Banach density 1. We remind the definition of lower Banach density.

\begin{definition}
Let \( B \subset \N \). Lower Banach density of \(B \), denoted by \( d_{*}(B) \) is
\[
d_{*}(B) = \liminf_{b-a \to \infty; a,b \in \N} \frac{\left| B \cap [a,b]  \right|}{b-a+1}.
\]
\end{definition}

\begin{theorem}
\label{theorem2} Let \( A \subset \N \) be a WM set, let \( B
\subset \N \) of positive upper density and let \( p(n) \in \Z[n]:
p(n) \to_{n \to \infty} \infty\) . Then the set \( R_p = \{ n \in
\N \, | \, p(n) \in A + B \}\) is of lower Banach density 1.
\end{theorem}

\noindent  We can generalize the result of theorem \ref{theorem2} and to prove the similar result for a number of different WM sets and different polynomials which have the same degree. Before stating the result we remind the notion of \textbf{essentially distinct} polynomials.

\begin{definition}
The polynomials \( \{p_1,\ldots,p_n \in \Z[n]\)\} are called \textbf{essentially distinct} if for every \( 1 \leq i < j \leq n \) we have \( p_i - p_j \) is a non-constant polynomial.
\end{definition}
All polynomials \( p(n) \) that we consider satisfy \( p(n) \to_{n \to \infty} \infty\).

\begin{theorem} 
\label{theorem3}
Let \( A \subset \N \) be a WM set, let \( p_1(n), \ldots,p_k(n) \in Z[n] \) be essentially distinct polynomials of the same degree, let \( B \subset \N \) of positive upper density. 
Then the set 
\[ R_{p_1,\ldots,p_k} = \{ n \in \N \, | \, \exists b \in B: \, p_1(n),  p_2(n), \ldots, p_k(n) \in A + b \} \] 
has lower Banach density 1.
\end{theorem} 

\begin{remark}
Any element \( n \in R_{p_1,\ldots,p_k} \) corresponds to a solution of the equation:
\begin{equation}
\label{additive_system} \left\{ \begin{array}{llll} x+y_1=p_1(n) \\
x+y_2 = p_2(n) \\
\ldots \\
x+y_k=p_k(n)
\end{array} \right.
\end{equation}
where \( x \in B, y_1, \ldots, y_k \in A \).
\end{remark}

\noindent There is an easy case which shows the necessity of some
restrictions on the degrees of the polynomials;  namely, when  there are two polynomials with
degrees which differ by at least two.

\begin{remark}
\label{remark_dif_at_least_two}
 \textnormal{If among \( p_1(n),\ldots,p_k(n) \)  there are two polynomials with degrees
which differ by at least two, then there exists a  WM set \( A \)  such that 
the set 
\[
R_{p_1,\ldots,p_k} = \{ n \in \N \, | \, \exists a \in A: \, p_1(n), p_2(n) , \ldots, p_k(n) \in A + b \}
\]
is empty.}
\end{remark}
\begin{proof}
We take an arbitrary WM set \( A \); then removing a set of
density zero from \( A \) leads again to a WM set (see definition \ref{def_TE_WM}). In particular, we can exclude from \( A \) all
solutions of the system
(\ref{additive_system}) by removing a set of density zero.
Namely, if \( \deg{p_1} \leq \deg{p_2} - 2 \) then replace \(
A \) by
\[
A' = A \setminus \left( \bigcup_{n \in \mathbb{N}}
[p_2(n)-p_1(n),p_2(n)] \right)
\]
which is again a WM set. (For sufficiently large \( n \) the polynomials \(
p_1(n), p_2(n) \) are monotone.) Within \( A' \) the system
(\ref{additive_system}) is unsolvable.

\hspace{12cm} \qed
\end{proof}

\noindent In the next sections we prove theorems \ref{theorem1} and \ref{theorem3}.

\noindent \textbf{Acknowledgment:} This work was done during my studies in Hebrew University of Jerusalem. I would like to thank my Ph.D. supervisor Hillel Furstenberg who introduced me to the subject of ergodic theory and proposed me the problem. I would like to thank Vitaly Bergelson for  fruitful discussions.

\section{Proof of theorem \ref{theorem1}}

\noindent Let \( A \) be a totally ergodic set (we don't require that density of \( A \) is positive). We introduce the normalized totally ergodic  sequence \( \xi \in \{-d(A), 1-d(A)\}^{\N}\) (\( d(A) \) is density of \( A \) ):  \( \xi(n) = 1_A(n) - d(A) \). Let \( p(n) \in \Z[n]: \deg{p} > 0, \, 
p(n) \to_{n \to \infty} \infty \). 

\noindent We use the following
\newline
\textbf{Notation:} \textit{The Hilbert space \( L^2(N) \) is the space of all real-valued
functions on the finite set \( \{1,2,\ldots,N\}\) endowed with the
following scalar product:
\[
\left<u,v\right>_{N} = \frac{1}{N} \sum_{n=1}^N u(n)v(n).
\]
We denote by \( \left\Vert u \right\Vert_{N} = \sqrt{\left<u,u\right>_{N}}\).}

\noindent The key tool to prove theorem \ref{theorem1} is the following lemma.

\begin{lemma}
\label{shift_ort}
For every \( \varepsilon > 0 \) there exists \( J(\varepsilon) \) such that for every \( J \geq J(\varepsilon) \) there exists \( N(J,\varepsilon) \) such that for every \( N \geq N(J,\varepsilon) \) we have 
\[
\left \Vert \frac{1}{J} \sum_{j=1}^J \xi(p(N+j) - n) \right \Vert_{p(N)} < \varepsilon.
\]
\end{lemma}


\noindent A proof of lemma \ref{shift_ort} relies on a standard technique introduced by V. Bergelson in his paper \cite{berg_pet}. A main ingredient is a finitary version of van der Corput lemma. At this stage we need the following simplified version of lemma \ref{vdrCorput}.

\begin{lemma}
\label{simple_corput}
Let \( \{u_{j}\}_{j=1}^{\infty } \) be a family of bounded vectors in a Hilbert space and let
\( \varepsilon >0 \). There exist \( J(\varepsilon), I(\varepsilon) \) such that If
\[
\left| \frac{1}{J}\sum ^{J}_{j=1}<u_{j}, u_{j+i}>\right| < \frac{\varepsilon}{2} \]
holds for
 \(J \geq J(\varepsilon) \) and every  \( 1\leq i\leq I(\varepsilon) \)
then
\[
\left\Vert \frac{1}{J}\sum _{j=1}^{J}u_{j}\right\Vert < \varepsilon.
\] 
\end{lemma}

\noindent First we prove a similar kind of result concerning polynomial shifts of a totally ergodic sequence.

\begin{lemma}
\label{imp_lemma}
Let \( A \subset \mathbb{N} \)  be totally ergodic set.
Let $p(x),q(x) \in \mathbb{Z}[x]$ be non constant polynomials with \( p(x),q(x) \to_{x \to \infty} \infty \)
 and $ \deg{q(x)} < \deg{p(x)}$,  then for any
$\varepsilon > 0$ and any $J'$ there exist  \( J(\varepsilon,J') \) with
\(  J(\varepsilon,J') \geq J' \) such that for every \(J \geq  J(\varepsilon,J') \) there exists
$N(\varepsilon,J)$ such that for every
\( N \geq N(\varepsilon, J) \)  we have
\[
\left\Vert \frac{1}{J}\sum _{j=1}^{J}v_{j}^q\right\Vert _{p(N)} <\varepsilon,
\]
where $v_j^q(n) = \xi (n  +  q(N+j))$;$1 \leq n \leq p(N)$.
\newline
(\( \xi(n) = 1_A(n) - d(A) \), where \( d(A) \) denotes the density of \( A \))
\end {lemma}

\begin{remark}
$N(\varepsilon,J)$ in the lemma is chosen to be such that $p(N)$
and $q(N)$ greater than zero for any \( N \geq N(\varepsilon,J)\).
\end{remark}
\begin{proof}
By induction on $\deg q(x)$.
\newline
For $ \deg q(x) = 1 $ the claim follows from
total ergodicity of $A$.
\newline
Assume \( q(x) = a x + b \) then

\[
 \left\Vert \frac{1}{J}\sum _{j=1}^{J}v_{j}^q\right\Vert _{p(N)}^2 =
\frac{1}{p(N)}\sum ^{p(N)}_{n=1} ( \frac{1}{J}\sum ^{J}_{j=1}  \xi(n + a (N+j) + b ) )^2 =
\frac{1}{p(N)}\sum ^{p(N)}_{n=1} ( \frac{1}{J}\sum ^{J}_{j=1}  \xi( n + a j ) )^2 -
\]

\[
 - \frac{1}{p(N)}\sum ^{a N + b}_{n=1} ( \frac{1}{J}\sum ^{J}_{j=1}  \xi( n + a j ) )^2 +
\frac{1}{p(N)}\sum ^{p(N) + a N + b}_{n=p(N)+1} ( \frac{1}{J}\sum ^{J}_{j=1}  \xi( n + a j ) )^2 =
\]

\[
\frac{1}{p(N)}\sum ^{p(N)}_{n=1} ( \frac{1}{J}\sum ^{J}_{j=1}  \xi( n + a j ) )^2 + \delta_{N,J},
\]
where  \(  \delta_{N,J} \rightarrow_{ N \rightarrow \infty} 0 \). By genericity of the point
\( \xi \in X_{ \xi} \) it follows that 
\begin{equation}
\label{crucial}
\frac{1}{p(N)}\sum ^{p(N)}_{n=1} ( \frac{1}{J}\sum ^{J}_{j=1}  \xi( n + a j ) )^2
\rightarrow_{ N \rightarrow \infty}  \int_{ X_{ \xi } }  ( \frac{1}{J}\sum ^{J}_{j=1}  f(T^{ aj } x) )^2 d\mu(x) =
 \left\Vert \frac{1}{J}\sum _{j=1}^{J}T^{ aj }f  \right\Vert_{L^2(X_{ \xi}, \mu)}^2
\end{equation}
where  \( f \in C( X_{ \xi } ) \) and it is defined by
\( f( \omega ) = \omega_1, \omega = ( \omega_1, \omega_2, \ldots , \omega_n, \ldots ) \). Note that 
\[
\int_{X_{\xi}} f d\mu = 0.
\]
\newline
Applying Von-Neumann ( \( L^2 \) ) ergodic theorem and by using ergodicity of \( ( X_{ \xi }, \mathbb{B}, \mu, T^a ) \)
we get
\[
\frac{1}{J}\sum _{j=1}^{J}T^{ aj }f  \rightarrow_{ J \rightarrow \infty }^{L^2} c
\]
where \( c \) is some constant. To prove  \( c = 0 \) we use the easy fact that if
\( g_J \rightarrow^{L^1}_{J \rightarrow \infty} 0 \)
and \(  \left\Vert g_J \right\Vert_{\infty} \leq M \) for any J, then
\( g_J \rightarrow^{L^2}_{J \rightarrow \infty} 0 \). 
The system \( ( X_{ \xi }, \mathbb{B},  \mu, T^a ) \) is ergodic thus by using Birkhoff ergodic theorem we get
\[  \frac{1}{J}\sum _{j=1}^{J}T^{ aj }f  \rightarrow_{ J \rightarrow \infty }^{L^1} \int_{X_{\xi}} f d\mu = 0. \]

\noindent
Let \( \varepsilon > 0 \) and \( J' \) is given. We showed that there exists \( J(\varepsilon) \)
such that for every \( J \geq J(\varepsilon) \) holds
\[
\left\Vert \frac{1}{J}\sum _{j=1}^{J}T^{ aj }f  \right\Vert_{L^2(X_{ \xi}, \mu)}^2 \leq \frac {\varepsilon} {2}.
\]
Define \( \mathbb{J} \risingdotseq \max( J(\varepsilon),J' ) \). Then from  (\ref{crucial}) it follows
that for  every \( J \geq  \mathbb{J} \) there exists \( N_1(\varepsilon,J) \) such that for every
\( N \geq  N_1(\varepsilon,J) \) we have

\[
\frac{1}{p(N)}\sum ^{p(N)}_{n=1} ( \frac{1}{\mathbb{J}}\sum ^{\mathbb{J}}_{j=1}  \xi( n + a j ) )^2  \leq \frac { 3 \varepsilon}{4}.
\]
On the other hand there exists \( N_2(\varepsilon,J) \) such that for every \( N \geq N_2(\varepsilon,J) \) we have 
\( \delta_{N,J} < \frac{\varepsilon}{4} \). 

\noindent Therefore for every
\( N \geq N(\varepsilon,J) \risingdotseq \max \left(N_1(\varepsilon,J),N_2(\varepsilon,J)\right) \)
we have

\[
 \left\Vert \frac{1}{J}\sum _{j=1}^{J}v_{j}^q\right\Vert _{p(N)}^2 < \varepsilon.
\]




\noindent Let  $ \deg q(x) = n $.

\noindent We have 
$\left\Vert v_j^q \right \Vert_{p(N)} \leq 1$, 
 therefore by lemma \ref{simple_corput} it is enough to show that there exists \( \mathbb{J}  \geq max(J',J(\varepsilon))\) such
that for every $J \geq \mathbb{J}$ there exists \( N(\varepsilon,J) \) such that for every
$N \geq N(\varepsilon,J)$ and every $ i: 1 \leq i \leq I(\varepsilon) $ we have 

\begin{equation}
\label{main_ineq}
\left| \frac{1}{J}\sum ^{J}_{j=1}<v_{j}^q, v_{j+i}^q>_{p(N)} \right| < \frac {\varepsilon} {2}.
\end{equation}
Let \( J'' = max( J',  J(\varepsilon ) ) \).

\noindent 
We have

\[
 \frac{1}{J}\sum ^{J}_{j=1}<v_{j}^q, v_{j+i}^q>_{p(N)} =
 \frac{1}{J}\sum ^{J}_{j=1} \frac{1}{p(N)}\sum ^{p(N)}_{n=1} \xi(n + q(N+j)) \xi(n + q(N+j+i)) =
\]

\[
= \frac{1}{J}\sum ^{J}_{j=1} \frac{1}{p(N)}\sum ^{p(N)}_{n=1} \xi(n) \xi(n + q(N+j+i) - q(N+j))
\]

\[
 - \frac{1}{J}\sum ^{J}_{j=1} \frac{1}{p(N)}\sum ^{q(N+j)}_{n=1} \xi(n) \xi(n + q(N+j+i) - q(N+j))
\]

\[
 + \frac{1}{J}\sum ^{J}_{j=1} \frac{1}{p(N)}\sum ^{p(N) + q(N+j)}_{n=p(N)+1} \xi(n) \xi(n + q(N+j+i) - q(N+j))
\]

\[
 = \frac{1}{p(N)}\sum ^{p(N)}_{n=1} \xi(n)  \frac{1}{J}\sum ^{J}_{j=1}  \xi(n + q(N+j+i) - q(N+j))
+ \delta_{N,J,i},
\]
where \( \delta_{N,J,i} \rightarrow_{N \rightarrow \infty} 0 \).
\newline
Denote
$w_{i,j}^q(n) = \xi (n + q(N+j+i) - q(N+j))$; $1 \leq n \leq p(N)$,
\( r(x) = q(x+i) - q(x) \). Note that \( \deg{r(x)} = \deg{q(x)} - 1 \) and \( r(x) \to_{x \to \infty} \infty \). By induction's hypothesis it follows
that there exists $J( J'',\frac{\varepsilon}{4},i)$ ( note \( J( J'',\frac{\varepsilon}{4},i) ) \geq J'' \) )
such that for every
\( J \geq J( J'',\frac{\varepsilon}{4},i ) \) there exists $N( \frac{\varepsilon}{4} , i , J )$  and we have 

\[
\left\Vert \frac{1}{J}\sum _{j=1}^{J}w_{i,j}^q\right\Vert _{p(N)}< \frac {\varepsilon} {4}
\]
for every \( N \geq  N(\frac {\varepsilon}{4}, i, J ) \). 
 Cauchy-Schwartz  inequality implies

\[
 \left| \frac{1}{J}\sum ^{J}_{j=1}<v_{j}^q, v_{j+i}^q>_{p(N)} \right|  \leq
  \left| <\xi , \frac{1}{J}\sum ^{J}_{j=1} w_{i,j}^q>_{p(N)} \right| + | \delta_{N,J,i}|
\]

\[
 \leq \left\Vert \xi \right\Vert _{p(N)}  \left\Vert \frac{1}{J}\sum _{j=1}^{J}w_{i,j}^q\right\Vert _{p(N)}
+ | \delta_{N,J,i}| = \frac {\varepsilon} {4} + | \delta_{N,J,i}|
\]
for every chosen $J \geq J(J'',\frac{\varepsilon}{4}, i )$ and every
$ N \geq N( \frac {\varepsilon}{4}, i , J ) $. 
We noted that \( \delta_{N,J,i} \to_{N \to \infty} 0 \). Therefore  there exists $N'(\frac{\varepsilon}{4}, i, J)$ such that for every $ N \geq
N'(\frac{\varepsilon}{4}, i, J)$ we have \( | \delta_{N,J,i}| <
\frac{\varepsilon}{4} \).

\noindent 
Denote  $J_{\varepsilon,i,J''} \risingdotseq J(  J'',\frac{\varepsilon}{4},
i   )$. Let \( J \geq  J_{\varepsilon,i,J''} \), denote by    
\newline
$ N_{\varepsilon,i, J} \risingdotseq \max{(N(\frac {\varepsilon}{4}, i, J ),N'(\frac{\varepsilon}{4}, i, J ) )} $. Then  for every 
\( J \geq  J_{\varepsilon,i,J''} \) and every \( N \geq N_{\varepsilon,i, J} \) we have
\[
 \left| \frac{1}{J}\sum ^{J}_{j=1}<v_{j}^q, v_{j+i}^q>_{p(N)} \right|  < \frac {\varepsilon} {2}.
\]
Finally for every $ J \geq \mathbb{J} \risingdotseq \max_{1 \leq i \leq I(\varepsilon)} ( J_{\varepsilon,i,J''} )$
and for every  $ N \geq N(\varepsilon,J) \risingdotseq \max_{1 \leq i \leq I(\varepsilon)} ( N_{\varepsilon,i, J} )$
the inequality
 (\ref{main_ineq}) holds for  every \( 1 \leq i \leq I(\varepsilon) \). 
 
 \hspace{12cm} \qed
\end{proof}




















\noindent
\textbf{Proof of lemma \ref{shift_ort}.}

\noindent Denote $u_j (n) = \xi ( p(N+j) - n )$;$1 \leq n \leq p(N)$.
\noindent
For \( \deg p(x) = 1 \) the claim follows from total ergodicity of  \( A \) and genericity of point \( \xi \) (the same argument as for the linear case of lemma \ref{imp_lemma}).





\noindent
For \( \deg p(x) > 1 \), we use  lemma \ref{simple_corput}.
Note that  \( \left\Vert u_j \right\Vert_{p(N)} \leq 1\). Let \( \varepsilon > 0 \). If we show that  for every
\( J \geq J(\varepsilon) \) (where \( J(\varepsilon)\) is taken from Van der Corput's lemma) there exists
\( N(\varepsilon,J) \) such that for  every \( 1 \leq i \leq I(\varepsilon) \) [ \( I(\varepsilon) \) is also taken
from the formulation of Van der Corput's lemma] and for every \( N \geq N(\varepsilon,J) \) holds

\[
\left| \frac {1} {J} \sum^{J}_{j=1} < u_j , u_{j+i} >_{p(N)} \right| < \frac {\varepsilon} {2},
\]
then, by lemma \ref{simple_corput}, for every \(  J \geq J(\varepsilon) \) there exists \(  N(\varepsilon,J) \) such that for every
\( N \geq N(\varepsilon,J) \) we have

\[
\left\Vert \frac{1}{J}\sum _{j=1}^{J}u_{j}\right\Vert _{p(N)} <\varepsilon.
\]
One knows

\begin{equation}
\label{stam}
 \frac {1} {J} \sum^{J}_{j=1} < u_j , u_{j+i} >_{p(N)}  =
 \frac{1}{p(N)}\sum ^{p(N)}_{n=1} \xi(n)  \frac{1}{J}\sum ^{J}_{j=1}  \xi(n + p(N+j+i) - p(N+j))
+ \delta_{N,J,i},
\end{equation}
where \(\delta_{N,J,i} \rightarrow_{N \rightarrow \infty} 0 \).
Denote \( q(x)\risingdotseq p(x+i) - p(x) \) (\( \deg{q(x)} < \deg{p(x)} \)), 
\( v^q_{j,i} (n) \risingdotseq \xi (n + q(N+j)) \), \( n = 1, \ldots , p(N) \).
Then by 
lemma \ref{imp_lemma} we have

\[
\left| \frac{1}{p(N)}\sum ^{p(N)}_{n=1} \xi(n)  \frac{1}{J}\sum ^{J}_{j=1}  \xi(n + p(N+j+i) - p(N+j)) \right| =
\left| < \xi, \frac {1} {J} \sum^J_{j=1} v^q_{j,i}>_{p(N)}  \right| \leq \frac {\varepsilon} {4}
\]
holds for every \( J \geq \mathbb{J}_i \), for some \( \mathbb{J}_i \geq J(\varepsilon) \), and every
\( N \geq N( \frac {\varepsilon} {4}, J, i) \).
\newline
From (\ref{stam}) it follows that for every \( J \geq \mathbb{J}_i  \geq J(\varepsilon) \) there exists
\( N'(\frac {\varepsilon} {4}, J, i) \) such that for every \( N \geq N'( \frac {\varepsilon} {4}, J, i) \)
we have

\[
\left| \frac {1} {J} \sum^{J}_{j=1} < u_j , u_{j+i} >_{p(N)} \right| < \frac {\varepsilon} {2}.
\]
Denote \( \mathbb{J} \risingdotseq \max_{ 1 \leq i \leq I(\varepsilon) } \mathbb{J}_i \),
then for every \( J \geq \mathbb{J} \geq J(\varepsilon) \) there exists
\( N(\varepsilon, J) \risingdotseq \max_{ 1 \leq i \leq I(\varepsilon) } N'(\frac {\varepsilon} {4}, J, i) \) such
that for every \( N \geq  N(\varepsilon, J) \) we have
\[
\left| \frac {1} {J} \sum^{J}_{j=1} < u_j , u_{j+i} >_{p(N)} \right| < \frac {\varepsilon} {2}
\]
for every \( 1 \leq i \leq I(\varepsilon) \).

 \hspace{12cm} \qed

\noindent
\textbf{Proof of theorem \ref{theorem1}.}
Denote \( c = \overline{d(B)} > 0, u_j (n) = \xi(p(N+j) - n); \, 1 \leq n \leq p(N), \, 1 \leq j \leq J \).
\newline
If \( (A+B) \cap \{ p(n) | n \in \N \} = \emptyset \) then \( \forall b \in B, \forall N \in \N, \forall j \in \N: p(N+j) - b \not \in A \). Thus 
\[
\left<1_B, \frac{1}{J} \sum_{j=1}^J u_j \right>_{p(N)} = \frac{1}{p(N)} \sum_{n=1}^{p(N)}
1_B(n) \frac{1}{J} \sum_{j=1}^{J} \xi(p(N+j) - n) =
\]  
\[
-d(A) \frac{\left| B \cap \{ 1,2,\ldots, p(N) \}\right|}{p(N)}.
\]
Therefore for infinitely many \( N\)'s we have
\[
\left | \left<1_B, \frac{1}{J} \sum_{j=1}^J u_j \right>_{p(N)} \right| \geq \frac{d(A) c}{2}.
\] 

\noindent Take \( \varepsilon = \frac{d(A) c}{4} \). By lemma \ref{shift_ort} there exists \( J \) and \( N(J) \) 
such that for every \( N \geq N(J) \) we have
\begin{equation}
\label{bound_eq}
\left | \left<1_B, \frac{1}{J} \sum_{j=1}^J u_j \right>_{p(N)} \right| < \frac{d(A) c}{4}.
\end{equation}
We have got a contradiction.

\noindent If we assume that density of \( B \) exists and positive, then by use of (\ref{bound_eq}) for  \( N \) sufficiently large \( (A+B) \cap \{ p(N+1),\ldots,p(N+J) \} \neq \emptyset \). Thus the set
\[
R_p = \{ (A+B) \cap \{p(n) | \, n \in \N \} 
\]  
is syndetic.
\hspace{12cm} \qed

\section{Orthogonality of polynomial shifts }\label{subsect_orthogonality_polynomial_shifts}


\noindent
The following lemma is essentially the main tool in the proof of  theorem \ref{theorem3}.
It is inspired by the analogous proposition 2.0.1 in  \cite{fish2}.
\begin{lemma}
\label{first_orth_lemma}
Let \( A \subset \mathbb{N} \) be a  WM set   and assume that \( p_1,\ldots,p_k \in \mathbb{Z}[n] \) are essentially distinct
 polynomials with positive leading coefficients.  We set \( \xi(n)
= 1_{A}(n) - d(A)  \) for non-negative \( n \) and zero for \( n \leq 0 \), and we assume
 \( q(n) \in \mathbb{Z}[n] \) with a
positive leading coefficient, \( \deg( q) \geq \max_{1 \leq i \leq k}\deg (p_i) \) and for every
\( i  \, : \, 1 \leq i \leq k \) such that \( \deg(p_i) = \deg(q) \) we have
 that the leading coefficient
of \( q(n) \) is bigger than that of \( p_i \).
Then for  every \( \varepsilon > 0 \) there exists \( J(\varepsilon) \)
such that for every \( J \geq J(\varepsilon) \) there exists \(
N(J,\varepsilon) \) such that for every \( N \geq N(J,\varepsilon)
\) we have
\[
\left\Vert \frac{1}{J} \sum_{j=1}^J a_{N+j}\xi(n - p_1(N+j)) \xi(n -
p_2(N+j)) \ldots \xi(n - p_{k}(N+j)) \right\Vert_{q(N)} <
\varepsilon
\]
for every \( \{a_n\} \in \{0,1\}^{\mathbb{N}} \).
\end{lemma}
\begin{proof}
We prove this statement by using an analog of Bergelson's PET
induction, see \cite{berg_pet}. Let \( F = \{p_1,\ldots,p_k\}\) be a
finite set of polynomials and assume that the largest of the degrees
of \( p_i \) equals \( d \). For every \( i \, : \, 1 \leq i \leq d
\) we denote by \( n_i \) the number of different groups of
polynomials of degree \( i \), where two polynomials \(
p_{j_1},p_{j_2} \) of  degree \( i \) are in the same group if and
only if they have the same leading coefficient. We will say that \(
(n_1,\ldots,n_d)\) is the \textit{characteristic vector} of \( F \).

\noindent We prove a more general statement than the statement of the
lemma.

 \noindent Let \( \mathcal{F}(n_1,\dots,n_d) \) be the
family of all finite sets of essentially distinct polynomials having
characteristic vector \( (n_1,\ldots,n_d)\). Consider the following
two statements:
\newline
\( L(k ; n_1,\ldots,n_d) \): 'For every \(
\{g_1,\ldots,g_{n_1},q_1,\ldots,q_l\} \in
\mathcal{F}(n_1,\ldots,n_d)\), where \( d \leq \deg(q) \), \( q \) is increasing faster
than any \( q_i, \, i: \, 1 \leq i \leq l \) (the exact statement is formulated in lemma) and  \(
g_1,\ldots,g_{n_1} \) are linear polynomials, and every \(
\varepsilon,\delta
> 0 \) there exists \( H(\delta,\varepsilon) \in \mathbb{N} \) such
that for every \( H \geq  H(\delta,\varepsilon) \) there exists \(
J(H,\varepsilon) \in \mathbb{N} \) such that for every \( J
\geq J(H,\varepsilon) \) there exists \(
N(J,H,\varepsilon) \in \mathbb{N}\) such that for every \( N
\geq N(J,H,\varepsilon) \) for a set of \(
\{h_1\ldots,h_k\} \in [1 \ldots H]^k \) of density at least \( 1 -
\delta \) we have
\[
\Vert \frac{1}{J} \sum_{j=1}^J a_{N+j} \prod_{i=1}^{n_1}
\prod_{\epsilon \in \{0,1\}^k} \xi(n - g_i(N+j) - \epsilon_1 h_1 -
\ldots - \epsilon_k h_k) \prod_{i=1}^{l} \xi(n-q_i(N+j))
\Vert_{q(N)} < \varepsilon,
\]
for every \( \{a_n\} \in \{0,1\}^{\mathbb{N}} \)'.
\newline
\( L(k;\overline{n_1,\ldots,n_i},n_{i+1},\ldots,n_d)\): '\( L(k;n_1,\ldots,n_d) \) is valid for any \( n_1,\ldots,n_i \)'.

\noindent  Lemma \ref{first_orth_lemma} is the special case \(
L(0; \overline{n_1,\ldots,n_d})\), where  \( d \leq \deg(q) \) and the polynomial
\( q \) is increasing faster than all polynomials in the given family of polynomials
which has the characteristic vector \((n_1,\ldots,n_d)\). In
order to prove the latter  it is enough to establish \( L(k;1) \,
, \, \forall k \in \mathbb{N} \cup \{0\}\), and to prove the
following implications:
\[
  S.1_d: \,  L(k+1;n_1,n_2,\ldots,n_d) \Rightarrow L(k;n_1+1,n_2,\ldots,n_d);
  \]
  \[ k,n_1,\ldots,n_{d-1} \geq 0, n_d \geq 1, d \geq 1
\]
\[
  S.2_{d,i}: \, L(0;\overline{n_1,\ldots,n_{i-1}},n_i,\ldots,n_d) \Rightarrow L(k;\underbrace{0,\ldots,0}_{i-1 \, zeros},n_i+1,n_{i+1},\ldots,n_d);
\]
\[
\hspace{2in} k;n_1,\ldots,n_{d-1} \geq 0,n_d \geq 1, d \geq i >1
\]
\[
  S.3_d: \, L(k;\overline{n_1,\ldots,n_d}) \Rightarrow
L(k;\underbrace{0,\ldots,0}_{d \, zeros},1), \,\,\, k \geq 0 \, , \,
d \geq 1
\]

\noindent We start with a proof of statement \( S.2_{d,i} \). Suppose that \( F \) is a finite set of essentially distinct
polynomials and assume that the characteristic vector of \( F \) equals
\newline
\( (\underbrace{0,\ldots,0}_{i-1 zeros}, n_i+1,n_{i+1},\ldots,n_d)
\). Fix any of the \( n_i + 1 \) groups of polynomials of degree
\( i \) and denote its polynomials by \( g_1,\ldots,g_m \). Denote
the remaining polynomials in \( F \) by \( q_1,\ldots,q_l \).
Because there are no linear polynomials among the polynomials of
\( F \) , we have to show the following:

\noindent \textit{Let the family \( F \doteq
\{g_1,\ldots,g_m,q_1,\ldots,q_l\}\) of polynomials with the
characteristic vector \( (\underbrace{0,\ldots,0}_{i-1
zeros},n_i+1,n_{i+1},\ldots,n_d)\), where \(\{g_1,g_2,\ldots,g_m\}
\in \mathbb{Z}[n] \) is one of the groups of \( F\) of the degree
\( i, \, i > 1 \). Let \( A \) be a WM set and denote by \( \xi \)
the normalized WM-sequence, i.e., \( \xi(n) = 1_A(n) - d(A) \, ,
\, \forall n \in \mathbb{N}\). For every \( \varepsilon, \delta
> 0 \) there exists \( H(\varepsilon,\delta) \in \mathbb{N} \)
such that for every \( H \geq H(\varepsilon,\delta) \) there
exists \( J(\varepsilon,H) \) such that for every \( J \geq
J(\varepsilon,H) \) there exists \( N(J,\varepsilon,H) \) such
that for every \( N \geq N(J,\varepsilon,H) \) for a set of \(
(h_1,\ldots,h_k) \in \{1,\ldots,H\}^k \) of density which is at
least \( 1 - \delta \) we have
\[
\Vert \frac{1}{J} \sum_{j=1}^J a_{N+j}\prod_{\epsilon \in \{0,1\}^k} \xi(n - \epsilon_1 h_1 -\ldots-\epsilon_k h_k) \xi(n - g_1(N+j))  \ldots
\xi(n - g_{m}(N+j))
 \]
  \[ \xi(n - q_{1}(N+j))\ldots
\xi(n - q_{l}(N+j))\Vert_{q(N)} < \varepsilon,
\]
for every \( \{a_n\} \in \{0,1\}^{\mathbb{N}}\) and with the
condition \( \deg(q) \geq d \) and \( q \) is increasing faster than any
\( q_i, \, i: \, 1 \leq i \leq l \).}

\noindent  Denote by
\[
 u_j(n) \doteq  a_{N+j}\xi(n - g_1(N+j))  \ldots \xi(n - g_{m}(N+j))
\]
\[
 \xi(n - q_{1}(N+j))\ldots \xi(n - q_{l}(N+j)) ,
\]
\[
w(n) = \prod_{\epsilon \in \{0,1\}^k} \xi(n - \epsilon_1 h_1
-\ldots-\epsilon_k h_k),
\]
\[
 v_j(n) = w(n) u_j(n),
\]
\[
  \hspace{2in}  n=1,\ldots,q(N).
\]
The sequence \( w(n) \) is bounded by \( 1 \) and therefore to prove
that \( \Vert \frac{1}{J} \sum_{j=1}^J v_j\Vert_{q(N)} \) is small
it is sufficient to show that  \( \Vert \frac{1}{J} \sum_{j=1}^J u_j
\Vert_{q(N)} \) is small.

\noindent We apply the van der Corput lemma (see lemma
\ref{vdrCorput} in appendix):

\[
   \frac{1}{J} \sum_{j=1}^J <u_j,u_{j+h}>_{q(N)} =
\]
\[
   \frac{1}{q(N)} \sum_{n=1}^{q(N)} \frac{1}{J} \sum_{j=1}^J
a_{N+j}\xi(n - g_1(N+j))  \ldots \xi(n-g_m(N+j))
\]
\[
\xi(n- q_1(N+j)) \ldots \xi(n-q_l(N+j))
\]
\[
a_{N+j+h}\xi(n - g_1(N+j+h))  \ldots \xi(n - g_{m}(N+j+h))
\]
\[
\xi(n - q_1(N+j+h)) \ldots \xi(n-q_l(N+j+h)) =
\]
\[
    \frac{1}{q(N)-g_1(N)} \sum_{n=1}^{q(N)} \xi(n) \frac{1}{J} \sum_{j=1}^J a_{N+j}a_{N+j+h} \xi(n - (g_2(N+j) - g_1(N+j)))
    \ldots
\]
\[
    \xi(n- ( g_m(N+j)-g_1(N+j) )) \xi(n- (q_1(N+j) - g_1(N+j))) \ldots
\]
\[
    \xi(n-(q_l(N+j)-g_1(N+j)))  \xi(n - (g_1(N+j+h) - g_1(N+j)))  \ldots
\]
\[
    \xi(n - (g_{m}(N+j+h) - g_1(N+j))) \xi(n - (q_1(N+j+h) - g_1(N+j))) \ldots
\]
\[
\xi(n-(q_l(N+j+h) - g_1(N+j))) + \delta_{N,J}=
\]
\[
   \frac{1}{q(N)} \sum_{n=1}^{q(N)-g_1(N)} \xi(n) \frac{1}{J} \sum_{j=1}^J b_{N+j}
   \xi(n - r_1(N+j)) \ldots \xi(n - r_{m-1}(N+j)) \xi(n - r_m(N+j)) \ldots
\]
\[
  \xi(n- r_{m+l-1}(N+j)) \xi(n - r_{m+l}(N+j)) \ldots \xi(n-r_{2m+l-1}(N+j))
\]
\[
  \xi(n-r_{2m+l}(N+j)) \ldots \xi(n - r_{2m+2l-1}(N+j)) + \delta_{N,J},
\]
where in the second equality we used a change of variable \( n
\leftarrow n=n-g_1(N+j) \), \( b_{N+j} = a_{N+j}a_{N+j+h} \), \(
\delta_{N,J} \rightarrow_{\frac{J}{N} \rightarrow 0} 0 \) and
\[
   \left\{ \begin{array}{llll} r_t(n) = g_{t+1}(n) - g_1(n) \, , \, t: \, 1 \leq t \leq m-1 \\
                               r_t(n) = q_{t-(m-1)}(n) - g_1(n) \, , \, t: \, m \leq t \leq m+l-1 \\
                               r_t(n) = g_{t-(m+l-1)}(n+h) - g_1(n) \, , \, t: \, m+l \leq t \leq 2m+l-1\\
                               r_t(n) = q_{t-(2m+l-1)}(n+h)-g_1(n) \, , \, t: \, 2m+l \leq t \leq 2m+2l-1.
\end{array} \right.
\]
For all but a finite number of \( h \)'s the polynomials \(
\{r_t(n)\}_{t=1}^{2m+2l-1} \) are essentially distinct, because \( i
> 1 \) and the polynomials \(g_1,\ldots,g_m,q_1,q_l\) are essentially
distinct. To see the last property we notice that if we take two
polynomials \( r_t\)'s from the same group (there are \( 4 \)
groups), then their difference is a
 non-constant
because the initial polynomials are essentially distinct. If we take
two polynomials from different groups then three cases are possible.
In the first case the difference of these polynomials is \( g_t(n+h)
-g_t(n) \) or \( q_t(n+h) - q_t(n) \) for some \( t\). We assume
that \( i > 1 \) therefore \( min_{1\leq t \leq l} \min( \deg (q_t),
\deg(g_1)) > 1 \) and from this it follows that  \( g_t(n+h) -g_t(n)
\) and \( q_t(n+h) - q_t(n) \) are non-constant polynomials. In the
second case we get for some \( t_1 \neq t_2 \): \( g_{t_1}(n+h) -
g_{t_2}(n) \) or \( q_{t_1}(n+h) - q_{t_2}(n) \). Here we note that
the map \( h \mapsto p(n+h) \) is an injective map from \(
\mathbb{N} \) to the set of essentially distinct polynomials, if \(
\deg(p) > 1 \). Thus, for all but a finite number of \( h \)'s we
get again a non-constant difference. In the third case we get for
some \( t_1,t_2 \): \(g_{t_1}(n+h) - q_{t_2}(n)\) or \( q_{t_1}(n+h)
- g_{t_2}(n)\). The resulting polynomial has the same degree as \(
q_t \).

\noindent  The characteristic vector of the set of polynomials \( \{r_1,\ldots,r_{2m+2l-1}\}\) has the form
\( (c_1,\ldots,c_{i-1},n_i,n_{i+1},\ldots,n_d)\). The polynomials from the second and the fourth group have the same
degree as \( q_t \) and the same leading coefficient as \( q_t \) if \( \deg(q_t) > \deg(g_1) \) and the leading
coefficient will be the difference of leading coefficients of \( q_t \) and \( g_1 \) if \( \deg(q_t) = \deg(g_1)\).
The polynomials from the first and the third group will be of degree smaller than \( \deg(g_1) \).

\noindent Applying \(
L(0;\overline{n_1,\ldots,n_{i-1}},n_i,\ldots,n_d) \) with the new polynomial
\( q(n)-g_1(n) \) which is increasing faster than all the polynomials
 \( \{r_t(n)\}_{t=1}^{2m+2l-1} \) and
the Cauchy-Schwartz inequality we get that for all but a finite number
of \( h \)'s and for every \( \varepsilon > 0 \) there exists \(
J(\varepsilon,h) \) such that for every \( J \geq J(\varepsilon,h)
\) there exists \( N(J,\varepsilon,h) \) such that for every \( N
\geq N(J,\varepsilon,h) \) we have
\[
  \left| \frac{1}{J} \sum_{j=1}^J <u_j,u_{j+h}>_{q(N)} \right| < \varepsilon,
\]
for every \( \{a_n\} \in \{0,1\}^{\mathbb{N}}\).

\noindent By the van der Corput lemma it follows that for every \(
\varepsilon > 0 \) there exists \( J(\varepsilon) \) such that for
every \( J \geq J(\varepsilon) \) there exists \( N(J,\varepsilon)
\) such that for every \( N \geq N(J,\varepsilon) \) we have
\[
  \left \Vert \frac{1}{J} \sum_{j=1}^J u_j \right \Vert_{q(N)} < \varepsilon,
\]
for every  \( \{a_n\} \in \{0,1\}^{\mathbb{N}}\). Thus we have shown
the validity of \( L(k;\underbrace{0,\ldots,0}_{i-1 zeros},
n_i+1,n_{i+1},\ldots,n_d) \).

\noindent We proceed with a proof of \( S.1_d \). We fix  the
 \( n_1 + 1 \) groups of the polynomials of degree \( 1 \) and denote its polynomials by
 \( g_1(n)=c_1n+d_1, \ldots, g_{n_1+1} = c_{n_1+1}n+d_{n_1+1}\).
 (By the assumption that
 all given polynomials are essentially distinct we get that in any group of degree \( 1 \) there is only one polynomial).
  The remaining  polynomials we denote by \( q_1,\ldots,q_l\). The set of polynomials
  \( \{g_1,\ldots,g_{n_1+1},q_1,\ldots,q_l\}\) has the characteristic vector \( ( n_1+1,n_2,\ldots,n_d)\).
  Again we  apply the van der Corput lemma. Let \( u_j(n) \)
  be defined as following
\[
 u_j(n) \doteq  a_{N+j}\prod_{i=1}^{n_1+1} \prod_{\epsilon \in \{0,1\}^k}\xi(n - g_i(N+j) - \epsilon_1 h_1 - \ldots - \epsilon_k h_k)
 \prod_{i=1}^l \xi(n - q_i(N+j)),
\]
\[
  \hspace{2in}  n=1,\ldots,q(N).
\]
Then we have
\[
  \frac{1}{J} \sum_{j=1}^J <u_j,u_{j+h}>_{q(N)} =
\]
\[
  \frac{1}{q(N)} \sum_{n=1}^{q(N)} \frac{1}{J} \sum_{j=1}^J a_{N+j} a_{N+j+h}
\]
\[
  \prod_{i=1}^{n_1+1} \prod_{\epsilon \in \{0,1\}^k}\xi(n - g_i(N+j) - \epsilon_1 h_1 - \ldots - \epsilon_k h_k)
 \prod_{i=1}^l \xi(n - q_i(N+j))
 \]
\[
 \prod_{i=1}^{n_1+1} \prod_{\epsilon \in \{0,1\}^k}\xi(n - g_i(N+j+h) - \epsilon_1 h_1 - \ldots - \epsilon_k h_k)
 \prod_{i=1}^l \xi(n - q_i(N+j+h))=
\]
\[
  \frac{1}{q(N)-g_1(N)} \sum_{n=1}^{q(N)} \prod_{\epsilon \in \{0,1\}^k} \xi(n - \epsilon_1 h_1 - \ldots - \epsilon_k h_k )
  \xi(n - \epsilon_1 h_1 - \ldots - \epsilon_k h_k - c_1 h)
\]
\[
  \frac{1}{J} \sum_{j=1}^J b_{N+j}
  \prod_{i=1}^{n_1} \prod_{\epsilon \in \{0,1\}^k} \xi(n-(c_{i+1}-c_1)(N+j) - (d_{i+1}-d_1) - \epsilon_1 h_1 - \ldots - \epsilon_k h_k)
\]
\[
\prod_{i=1}^{n_1} \prod_{\epsilon \in \{0,1\}^k}
\xi(n-(c_{i+1}-c_1)(N+j) - (d_{i+1}-d_1) - \epsilon_1 h_1 -
\ldots - \epsilon_k h_k - c_{i+1}h)
\]
\[
 \prod_{i=1}^l \xi(n-(q_i(N+j)-g_1(N+j))) \prod_{i=1}^l \xi(n-(q_i(N+j+h)-g_1(N+j))) +\delta_{N,J},
\]
where in the second equality we made a change of variable \( n
\leftarrow n - g_1(N+j) \) and  \( b_{N+j} = a_{N+j}a_{N+j+h} \, ,
\, \delta_{N,J} \rightarrow_{\frac{J}{N} \rightarrow 0} 0 \).

\noindent Denote by \( r_i(n) = (c_{i+1}-c_1)n + (d_{i+1}-d_1) \, , \, i: 1 \leq i \leq n_1 \),
\(  s_i(n) = q_i(n) - g_1(n) \, , \, t_i(n) = q_i(n+h) - g_1(n) \, , \, i: 1 \leq i \leq l \). Then the last expression may
be rewritten as
\[
  \frac{1}{q(N)-g_1(N)} \sum_{n=1}^{q(N)-g_1(N)} \prod_{\epsilon \in \{0,1\}^k} \xi(n - \epsilon_1 h_1 - \ldots - \epsilon_k h_k)
  \xi(n - \epsilon_1 h_1 - \ldots - \epsilon_k h_k- c_1 h)
\]
\[
  \frac{1}{J} \sum_{j=1}^J b_{N+j} \prod_{i=1}^{n_1}  \prod_{\epsilon \in \{0,1\}^k}
  \xi(n-r_i(N+j) - \epsilon_1 h_1 - \ldots - \epsilon_k h_k)
\]
\[
  \xi(n-r_i(N+j) - \epsilon_1 h_1 - \ldots - \epsilon_k h_k -c_{i+1}h)
\]
\[
  \prod_{i=1}^l \xi(n - s_i(N+j)) \xi(n-t_i(N+j)) + \delta_{N,J}  \doteq E1 + \delta_{N,J}.
\]
For every \( i \, : \, 1 \leq i \leq l \) the polynomials \( s_i,t_i
\) are in the same group (have the same degree and the same leading
coefficient), therefore the characteristic vector of the family \(
\{s_1,t_1,\ldots,s_l,t_l\}\) is the same as of the family \(
\{s_1,s_2,\ldots,s_l\}\) and , obviously, the characteristic vector
of the latter family is the same as of the family \(
\{q_1,q_2,\ldots,q_l\}\) and is equal to \((0,n_2,n_3,\ldots,n_d)\).
Again the polynomial \( q(n)-g_1(n) \) is increasing faster than any polynomial in the family
\(\{s_1,t_1,\ldots,s_l,t_l\}\) .
By use of \( L(k+1;n_1,\ldots,n_d) \) and the Cauchy-Schwartz
inequality we show that \( |E1| \) is arbitrarily small for a set of
arbitrarily large density of \( (h_1,\ldots,h_k,h)\)'s. Therefore,
by the van der Corput lemma we deduce the  validity of \(
L(k;n_1+1,n_2,\ldots,n_d)\).

\noindent The proof of \( S.3_d \) goes exactly in the same way as that
of \( S.2_{d,i} \).

\noindent \textbf{Proof of \( L(k;1) \, , \, \forall k \in \mathbb{N} \cup
\{0\}\)}:

\noindent  Assume that \( g_1(n) = c_1 n + d_1 \, , \, c_1 > 0 \) and \(q\) is increasing faster than \(g_1\) (\(q(n)-g_1(n) \to_{n \to \infty} \infty\)).
We show that

\noindent \textit{For every \( \varepsilon, \delta > 0\) there exists
\( H(\delta,\varepsilon) \in \mathbb{N} \) such that
for every \( H \geq H(\delta,\varepsilon) \) there exists \( J(H,\varepsilon) \in \mathbb{N} \) such that for every
\( J \geq J(H,\varepsilon)\) there exists \( N(J,H,\varepsilon) \) such that for every \( N \geq N(J,H,\varepsilon) \)
we have for a set of \( (h_1,\ldots,h_k) \in \{1,\ldots,H\}^k \) of density which is at least \( 1- \delta \) the
following
\[
  \left \Vert \frac{1}{J} \sum_{j=1}^J a_{N+j} \prod_{\epsilon \in \{0,1\}^k}
  \xi(n - g_1(N+j) - \epsilon_1 h_1 - \ldots - \epsilon_k h_k)\right \Vert_{q(N)} < \varepsilon
\]
for every \( \{a_n \} \in \{0,1\}^{\mathbb{N}} \).
}

\noindent We recall that to a WM set \( A \) is associated the weakly-mixing system \( (X_{\xi},\mathbb{B},T,\mu)\), where \( \xi(n) = 1_A(n) - d(A) \).
We define
the function \( f \) on \( X_{\xi}\)  by the following rule:
 \( f(\omega) = \omega_0 \, , \, \omega = \{\omega_0,\ldots,\omega_n,\ldots\} \in X_{\xi}\). It is evident that \( f \)
is continuous and \( \int_{X_{\xi}} f(x) d\mu(x) = 0 \). By genericity of the point \( \xi \in X_{\xi} \) we get
\[
\frac{q(N)}{q(N)-g_1(N)}\left \Vert \frac{1}{J} \sum_{j=1}^J a_{N+j} \prod_{\epsilon \in \{0,1\}^k}
 \xi(n - g_1(N+j) - \epsilon_1 h_1 - \ldots - \epsilon_k h_k)\right \Vert_{q(N)}^2 \rightarrow_{N \rightarrow \infty}
\]
\begin{equation}
\label{wm_discr_cube}
  \int_{X_{\xi}}  \left( \frac{1}{J} \sum_{j=1}^J a_{N+J+1-j}  T^{c_1 j}
  \left( \prod_{\epsilon \in \{0,1\}^k} T^{\epsilon_1 h_1 +\ldots +\epsilon_k h_k} f(x) \right)\right)^2 d\mu(x).
\end{equation}
Denote by \( g_{h_1,\ldots,h_k} \) the following function on \( X_{\xi}\):
\[
  g_{h_1,\ldots,h_k}(x) = \prod_{\epsilon \in \{0,1\}^k} T^{\epsilon_1 h_1 + \ldots + \epsilon_k h_k} f(x) \, ,
  \, \forall x \in X_{\xi}.
\]
Then we use the following statement which can be viewed as a corollary of theorem \( 13.1 \) of Host and Kra in \cite{host-kra}
(\( \int_{X_{\xi}} f(x) d\mu(x) = 0 \)).

\noindent \textit{For every \( \varepsilon, \delta > 0 \) there exists \( H(\delta,\varepsilon) \in \mathbb{N} \)
 such that for every
\( H \geq H(\delta,\varepsilon) \) for a set of \( (h_1,\ldots,h_k) \in \{1,\ldots,H\}^k\) which has density at least
\( 1- \delta \) we have
\[
   \left| \int_{X_{\xi}} g_{h_1,\ldots,h_k}(x) d\mu(x) \right| < \varepsilon.
\]}

\noindent

\noindent Let \( \varepsilon, \delta  > 0 \). By the foregoing
statement there exists \( H(\delta,\varepsilon) \in \mathbb{N} \)
such that for every \( H \geq H(\delta,\varepsilon)  \) the set of
those \( (h_1,\ldots,h_k) \in \{1,\ldots,H\}^k \) such that
\[
  \left| \int_{X_{\xi}} g_{h_1,\ldots,h_k}(x) d\mu(x) \right| < \frac{\varepsilon}{4}
\]
has density at least \( 1 - \delta \).

\noindent For any fixed \( \{h_1,\ldots,h_k\}\) lemma \ref{wm_sub_lemma1} implies that there exists
\( J(\varepsilon) \in \mathbb{N} \) such that for every \( J \geq  J(\varepsilon)  \)
 we have
\[
  \left\Vert \frac{1}{J} \sum_{j=1}^J b_j T^{c_1 j} \left( g_{h_1,\ldots,h_k}(x) - \int_{X_{\xi}} g_{h_1,\ldots,h_k}(x) d\mu(x)\right) \right \Vert_{L^2(X_{\xi})} < \frac{\varepsilon}{4}
\]
for any sequence \( \{b_n\} \in \{0,1\}^{\mathbb{N}} \).

\noindent Therefore, by merging  the two last statements we conclude that
there exists \( H(\delta,\varepsilon) \in \mathbb{N} \) such that
for every \( H \geq H(\delta,\varepsilon) \)  there exists \(
J(H,\varepsilon) \in \mathbb{N} \) such that for every \( J \geq
J(H,\varepsilon) \)  and for a set of \( (h_1,\ldots,h_k) \in
\{1,\ldots,H\}^k \) which has density at least \( 1 - \delta \) we
have
\[
  \left\Vert \frac{1}{J} \sum_{j=1}^J b_j T^{c_1 j} g_{h_1,\ldots,h_k}(x)\right \Vert_{L^2(X_{\xi})} < \frac{\varepsilon}{2}
\]
for any sequence \( \{b_n\} \in \{0,1\}^{\mathbb{N}} \).

\noindent Finally, by use of (\ref{wm_discr_cube}), the fact that
\( \lim_{N \rightarrow \infty} \frac{q(N)}{q(N)-g_1(N)} > 0 \)  and the last
statement we deduce the validity of \( L(k;1) \).

 \hspace{12cm} \qed
\end{proof}

\noindent The next lemma is a simple consequence of the previous one
and is used in the next section to prove  theorem \ref{theorem3}.
\begin{lemma}
\label{main_lemma}
 Let \( A \subset \mathbb{N} \) be a WM set and \( p_1,\ldots,p_k \in \mathbb{Z}[n] \) are essentially distinct
 polynomials of the same degree \( d \geq 1\), with positive leading coefficients such
  that \( p_1(n) > p_i(n) , \, \forall 1 < i \leq k\) for sufficiently large \( n \). Then
for every \( \varepsilon > 0 \) there exists \( J(\varepsilon) \)
such that for every \( J \geq J(\varepsilon) \) there exists \(
N(J,\varepsilon) \) such that for every \( N \geq N(J,\varepsilon)
\) we have
\[
\left\Vert \frac{1}{J} \sum_{j=1}^J a_{N+j}\xi( p_1(N+j) - n) \xi(
p_2(N+j) - n) \ldots \xi(p_{k}(N+j)-n) \right\Vert_{p_1(N)} <
\varepsilon
\]
for every \( \{a_n\} \in \{0,1\}^{\mathbb{N}} \), where \( \xi(n) =
1_A(n) - d(A) \) for non-negative \( n \)'s and zero for \( n \leq 0 \).
\end{lemma}
\begin{proof}
For a family of polynomials \( F = \{p_1,\ldots,p_k\}\) with a
maximal degree \( d \) denote by \( n_d \) the number
of different leading coefficients of polynomials of degree \( d \) from the family \( F \).
%
%

\noindent  As in the proof of lemma \ref{first_orth_lemma} we fix one of the groups of polynomials of degree \( d \) (all polynomials in the same group have the same
 leading coefficient). Assume that the group \(\{ g_1,\ldots,g_m \}\) has the maximal leading coefficient among all polynomials
 \( p_1,\ldots,p_k \). The rest of the polynomials we denote by \( q_1,\ldots,q_l \). Without loss of generality assume that
 \( p_1 = g_1, \ldots, p_m = g_m \).
Denote by
\( u_j(n) \, \, , \, \, 1 \leq n \leq p_1(N) \) the following
expression
\[
u_j(n) = a_{N+j}\xi( p_1(N+j) - n) \xi( p_2(N+j) - n) \ldots
\xi(p_{k}(N+j)-n).
\]
For \( u_j \)'s  we get
 \[
\frac{1}{J} \sum_{j=1}^J <u_j,u_{j+h}>_{p_1(N)} = \frac{1}{p_1(N)}
\sum_{n=1}^{p_1(N)} \frac{1}{J} \sum_{j=1}^J a_{N+j}\xi( p_1(N+j) -
n) \ldots \]
\[
\xi(p_{k}(N+j)-n) a_{N+j+h}\xi( p_1(N+j+h) - n) \ldots
\xi(p_{k}(N+j+h)-n) =
\]
\[
 \frac{1}{p_1(N)}  \sum_{n=1}^{p_1(N)} \xi(n) \frac{1}{J} \sum_{j=1}^J b_{N+j} \prod_{i=1}^{m-1} \xi(n-( p_1(N+j) - p_{i+1}(N+j)))
\]
\[
 \prod_{i=1}^{l}\xi(n - ( p_1(N+j) - q_i(N+j))) \prod_{i=1}^m  \xi(n-(p_1(N+j) - p_{i}(N+j+h)))
 \]
 \[ \prod_{i=1}^{l}\xi(n - (p_1(N+j) - q_i(N+j+h)))
 +\delta_{J,N},
\]
where \( b_{n} = a_n a_{n+h} \) and \( \delta_{J,N} \rightarrow_{\frac{J}{N} \rightarrow 0} 0 \).

\noindent Denote by \( r_i(n) =  p_1(n) - q_i(n) \, ; \, s_i(n) =  p_1(n) - q_i(n+h)\, , \, i: 1 \leq i \leq l \) and
\( t_i(n)  = p_1(n) - p_i(n) \, ; \, f_i(n) = p_1(n) - p_i(n+h) \, , \, i: 1 \leq i \leq m \). Then for all
 but a finite number of \( h \)'s  the polynomials
 \newline
 \( \tilde{F} \doteq \{r_1,\ldots,r_l,s_1,\ldots,s_l,t_2,\ldots,t_m,f_1,\ldots,f_m \} \)
 are essentially distinct and \(p_1\) is increasing faster than any polynomial in \( \tilde{F}\).
 Therefore by lemma \ref{first_orth_lemma} for all but a finite number of \( h \)'s
 the following expression is as small as we wish for appropriately chosen \( J,N \).
\[
  \Vert \frac{1}{J} \sum_{j=1}^J b_{N+j} \prod_{i=1}^{m-1} \xi(n-t_{i+1}(N+j))  \prod_{i=1}^{l}\xi(n - r_i(N+j))
  \]
  \[
  \prod_{i=1}^m  \xi(n-f_i(N+j)) \prod_{i=1}^{l}\xi(n - s_i(N+j)) \Vert_{p_1(N)}.
\]
Finally by Cauchy-Schwartz inequality and van der Corput's lemma
we get the desired conclusion.

\hspace{12cm} \qed
\end{proof}

\section{Proof of theorem \ref{theorem3} 
}\label{subsect_proof_suff_theorem_main_thm}


\textbf{Proof of  theorem \ref{theorem3}.}


\noindent Assume we have an arbitrary WM set \( A \) and \( k \)
essentially distinct polynomials \( p_1,\ldots,p_k \in
\mathbb{Z}[n] \)  of the same degree \( d \geq 1\) with  positive
leading coefficients and assume that for sufficiently large \( n
\)'s we have \( p_1(n) > p_i(n) \, , \, \forall i: \,  2 \leq i
\leq k \). Let's define the set \( F \) of all \( z \)'s where the
statement of the theorem fails, namely,
\[
F \risingdotseq \{ z \in \mathbb{N} \, | \, for \,\, any \,\,
(x,y_1,\ldots,y_k) \in A^{k+1} \,\, the \,\, system
\,\,\text{(\ref{additive_system})} \,\, fails\,\, to \,\,
hold\}.
\]
We shall prove that \( d^{*}(F) = 0 \). Since \( d(A) > 0 \) we
can find \( z \in A, z \not \in F \) and this will yield a
solution to (\ref{additive_system}).

\noindent Denote by \( \{a_n\} \) the indicator sequence of \( F
\), i.e., \( a_n = 1_F(n) \). We define the sequence \( \xi \) to
be a normalized indicator sequence of \( A \): \( \xi(n) = 1_A(n)
- d(A) \, , \, n \in \mathbb{N} \) and zero for non-positive values of \( n \),
 where \( d(A) \) is the density of \( A \) which
exists.
\newline
We define the expression \(B_{N,J} \) to be
\begin{equation}
\label{def_B_N_J}
B_{N,J} \risingdotseq \frac{1}{p_1(N)} \sum_{n=1}^{p_1(N)}
\frac{1}{J} \sum_{j=1}^J a_{N+j} 1_A(n) 1_A(p_1(N+j) - n)
\end{equation}
\[
1_A(p_2(N+j) - n) \ldots 1_A(p_{k-1}(N+j)-n) \xi(p_k(N+j)-n).
\]

\noindent     Suppose that we have \( d^{*}(F) > 0\). Then there
exist intervals \( I_{l,J} = [u_{l,J}+1,u_{l,J} + J] \) (for \( J
\) big enough) such that \( u_{l,J} \rightarrow _{l \rightarrow
\infty} \infty \) and \( \frac{|F \cap I_{l,J}|}{J}
> \frac{d^{*}(F)}{2} \) for every \(l\) and \(J\) big enough.
By induction on \( k \) and \( i \) we  prove the validity of the following
claim.

\noindent \textbf{Claim 1:} \textit{For every \(i \, : \, 0 \leq i
\leq k-1\) and every \( \varepsilon > 0 \) there exist \(J,l \) big
enough such that
\[
|\frac{1}{p_1(u_{l,J})} \sum_{n=1}^{p_1(u_{l,J})} \frac{1}{J}
\sum_{j=1}^J b_{u_{l,J}+j} 1_A(n) 1_A(p_1(u_{l,J}+j) - n)\ldots
\]
\[
1_A(p_i(u_{l,J}+j) - n) \xi(p_{i+1}(u_{l,J}+j) - n)
 \ldots \xi(p_{k}(u_{l,J}+j)-n) | <
 \varepsilon
\]
for every \(\{0,1\}\)-valued sequence \(\{b_n\}\). }

\noindent A proof of claim \( 1 \) is by induction on \( i \) and \( k \).
\newline
In the sequel we use the notation \( \left<1_A,f(n)\right>_N \), where \( f(n) \) is defined for all
\( n=1,2,\ldots,N \); which has the same meaning
as \(\left <1_A,f\right>_N = \frac{1}{N} \sum_{n=1}^N 1_A(n) f(n)\).
\newline
For \( i = 0 \) and every \( k \) the statement is exactly of lemma \ref{main_lemma}. For every
\( i < k-1 \) we
will prove the statement of the claim for \( i +1 \) and \( k \) provided
the statement for \( i \) and \( k \), and for \( i \), \( k - 1\):
\[
|\frac{1}{p_1(u_{l,J})} \sum_{n=1}^{p_1(u_{l,J})} \frac{1}{J}
\sum_{j=1}^J b_{u_{l,J}+j} 1_A(n) 1_A(p_1(u_{l,J}+j) - n)\ldots
\]
\[
1_A(p_i(u_{l,J}+j) - n) 1_A(p_{i+1}(u_{l,J}+j) - n)
\xi(p_{i+2}(u_{l,J}+j) - n) \ldots \xi(p_{k}(u_{l,J}+j)-n)
| =
\]
\[
| <1_A,\frac{1}{J} \sum_{j=1}^J b_{u_{l,J}+j}
1_A(p_1(u_{l,J}+j) - n) \ldots
\]
\[1_A(p_i(u_{l,J}+j) - n)
(\xi(p_{i+1}(u_{l,J}+j) - n) + d(A)) \xi(p_{i+2}(u_{l,J}+j) - n)
\ldots
\]
\[\xi(p_{k}(u_{l,J}+j)-n)
>_{p_1(u_{l,J})}| \leq
\]
\[
| <1_A,\frac{1}{J} \sum_{j=1}^J b_{u_{l,J}+j}
1_A(p_1(u_{l,J}+j) - n) \ldots
\]
\[1_A(p_i(u_{l,J}+j) - n)
\xi(p_{i+1}(u_{l,J}+j) \xi(p_{i+2}(u_{l,J}+j) - n)
\ldots
\]
\[\xi(p_{k}(u_{l,J}+j)-n)
>_{p_1(u_{l,J})}| +
\]
\[
d(A) | <1_A,\frac{1}{J} \sum_{j=1}^J b_{u_{l,J}+j}
1_A(p_1(u_{l,J}+j) - n) \ldots
\]
\[1_A(p_i(u_{l,J}+j) - n)
 \xi(p_{i+2}(u_{l,J}+j) - n)
\ldots
\]
\[\xi(p_{k}(u_{l,J}+j)-n)
>_{p_1(u_{l,J})}| < \varepsilon,
\]
for big enough \(J,l\).
The first summand is small by the statement of the claim for \( i \) and \( k \), and
the second summand is small by the statement of the claim for \( i \) and \( k -1 \).
This ends the proof of claim \( 1 \).

\noindent  We will use the statement of claim \( 1 \) for \( i = k-1 \) and we call the statement claim \( 2 \).

\noindent \textbf{Claim 2:} \textit{For every \(\varepsilon > 0 \)
there exist \( J,l \) big enough such that  the expression
\[
|\frac{1}{p_1(u_{l,J})} \sum_{n=1}^{p_1(u_{l,J})} \frac{1}{J}
\sum_{j=1}^J b_{u_{l,J}+j} 1_A(n) 1_A(p_1(u_{l,J}+j) - n) \ldots\]
\[
  1_A(p_{k-1}(u_{l,J}+j)-n) \xi(p_k(u_{l,J}+j)-n) | <
 \varepsilon
\]
for every \{0,1\}-valued sequence \(\{b_n\}\). }

\noindent The next statement enables us to conclude about a boundedness away from zero
of \( B_{u_{l,J},J} \).

\noindent \textbf{Claim 3:} \textit{For every \( \delta > 0 \)
for big enough \( J,l \)  the expression
\[
\frac{1}{p_1(u_{l,J})} \sum_{n=1}^{p_1(u_{l,J})} \frac{1}{J}
\sum_{j=1}^J b_{u_{l,J}+j} 1_A(n) 1_A(p_1(u_{l,J}+j) - n)
 \ldots 1_A(p_{k}(u_{l,J}+j)-n)
\]
is bigger than \( c (1 - \delta) d^{k+1}(A) \frac{d^{*}(F)}{3} \),
where \( c = \min_{2 \leq i \leq k-1} \frac{c_i}{c_1} \) (\( c_i
\) is a leading coefficient of polynomial \( p_i \)) for every
\(\{0,1\}\)-valued sequence \(\{b_n\}\) which has density bigger than
\(\frac{d^{*}(F)}{2} \) on all intervals \( I_{l,J}\).}

\noindent  The proof is by induction on \( k \).

\noindent For \( k = 1 \)  by using lemma \ref{main_lemma} we
have that for \( J \) and \( l \) big enough
\[
 \frac{1}{p_1(u_{l,J})} \sum_{n=1}^{p_1(u_{l,J})} \frac{1}{J}
\sum_{j=1}^J b_{u_{l,J}+j} 1_A(n) 1_A(p_1(u_{l,J}+j) - n) =
\]
\[
<1_A,\frac{1}{J} \sum_{j=1}^J b_{u_{l,J}+j}(\xi(p_1(u_{l,J}+j) -
n) + d(A))>_{p_1(u_{l,J})} \,\, \geq
\]
\[ - \varepsilon +
d(A)<1_A,\frac{1}{J} \sum_{j=1}^J b_{u_{l,J}+j}>_{p_1(u_{l,J})} \,\,
> (1-\delta) d(A)^2 \frac{d^{*}(F)}{3}.\]

\noindent
Assume the statement of the claim holds for \( k \).  Let
\((p_1,\ldots,p_{k},p_{k+1})\) be polynomials of the same degree such that \(p_1 \) is the
 ``biggest" among them (see conditions of lemma \ref{main_lemma}).
Without loss of generality we can assume that \( \min_{2 \leq i
\leq k+1} {c_i} = c_{k+1}\).  Then
for sufficiently large \( J\) and \(l\)
\[
\frac{1}{p_1(u_{l,J})} \sum_{n=1}^{p_1(u_{l,J})} \frac{1}{J}
\sum_{j=1}^J b_{u_{l,J}+j} 1_A(n) 1_A(p_1(u_{l,J}+j) - n) \ldots
\]
\[
1_A(p_k(u_{l,J}+j) - n) 1_A(p_{k+1}(u_{l,J}+j) - n) =
\]
\[
<1_A,\frac{1}{J} \sum_{j=1}^J b_{u_{l,J}+j} 1_A(p_1(u_{l,J}+j) -
n) \ldots
\]
\[ 1_A(p_k(u_{l,J}+j) - n) (\xi(p_{k+1}(u_{l,J}+j) - n) +
d(A)) >_{p_1(u_{l,J})}  -
\]
\[
d(A) \frac{1}{p_1(u_{l,J})} \sum_{n= p_{k+1}(u_{l,J})}^{p_1(u_{l,J})}
1_A(n) \frac{1}{J} \sum_{j=1}^J b_{u_{l,J}+j} 1_A(p_1(u_{l,J}+j) -  n) \ldots
  1_A(p_k(u_{l,J}+j) - n) =
\]
\[
d(A)<1_A,\frac{1}{J} \sum_{j=1}^J b_{u_{l,J}+j} 1_A(p_1(u_{l,J}+j)
- n) \ldots 1_A(p_k(u_{l,J}+j) - n)>_{p_1(u_{l,J})}+
\]
\[
<1_A,\frac{1}{J} \sum_{j=1}^J b_{u_{l,J}+j}1_A(p_1(u_{l,J}+j) - n)
\ldots 1_A(p_k(u_{l,J}+j) - n) \xi(p_{k+1}(u_{l,J}+j) -
n)>_{p_1(u_{l,J})}  -
\]
\[
 d(A) \frac{1}{p_1(u_{l,J})} \sum_{n= p_{k+1}(u_{l,J})}^{p_1(u_{l,J})}
1_A(n) \frac{1}{J} \sum_{j=1}^J b_{u_{l,J}+j} 1_A(p_1(u_{l,J}+j) -  n) \ldots
  1_A(p_k(u_{l,J}+j) - n)
\]
\[
  >
\]
\[
   d(A) \frac{1}{p_1(u_{l,J})} \sum_{n= 1}^{p_{k+1}(u_{l,J})-1}
1_A(n) \frac{1}{J} \sum_{j=1}^J b_{u_{l,J}+j} 1_A(p_1(u_{l,J}+j) -  n) \ldots
  1_A(p_k(u_{l,J}+j) - n)
   - \varepsilon
\]
\[
  > d(A)c(1-\delta')d(A)^{k+1}\frac{d^*(F)}{3}
\]
\[
  > c(1 - \delta)d(A)^{k+2}\frac{d^{*}(F)}{3}.
\]
We used claim \( 2 \) in the first inequality and induction hypothesis in
the second inequality. This ends the proof of claim \( 3 \).

%

 \noindent  By the
definition of \( F \)  it follows that for every non-zero value of
\[
a_{u_{l,J}+j} 1_A(n) 1_A(p_1(u_{l,J}+j) - n) 1_A(p_2(u_{l,J}+j) -
n) \ldots 1_A(p_{k-1}(u_{l,J}+j)-n) \]
 (thus it equals to one), the remaining factor in the summands of \( B_{u_{l,J},J} \)
 is negative, namely, \( \xi(p_k(u_{l,J}+j)-n) = -d(A)
\). Therefore, by using claim \( 3 \) we get \( |B_{u_{l,J},J}| \geq c (1 - \varepsilon) d^{k+1}(A)
\frac{d^{*}(F)}{3} \) for any \( l \) and for \( J \) big
enough. Thus \( |B_{u_{l,J},J}| \) is bounded from zero.
\newline
On the other hand, by claim \( 2 \)  it follows
that for any \( \varepsilon > 0 \) there exists \( J= J(\varepsilon)
\) and \( N = N(J(\varepsilon))\) such that \( |B_{N,J}| <
\varepsilon \). Therefore we get a contradiction.

\noindent    We have proved that the set of all \( z\)'s such that the
system (\ref{additive_system}) is solvable within \( A^{k+1}\) (\( z \) is not necessarily in \( A\)) has a lower density one. Therefore it
intersects every set of positive density (even of positive upper
density), in particular, \(A\).

\hspace{12cm} \qed

\section{Appendix}\label{appendix}
\numberwithin{lemma}{section} \numberwithin{theorem}{section}
\numberwithin{prop}{section} \numberwithin{remark}{section}
\numberwithin{definition}{section} \numberwithin{corollary}{section}

\begin{lemma}\( \rm( \)\textit{van der Corput}\( \rm) \)
\label{vdrCorput}
 Suppose \(\varepsilon
>0 \) and \( \{u_{j}\}_{j=1}^{\infty } \) is a family
of vectors in  Hilbert space, such that \( \Vert u_j \Vert \leq 1 \,
\rm( 1 \leq j \leq \infty \rm)  \). Then there exists \(
I'(\varepsilon) \in \mathbb{N} \), such that for every \( I \geq
I'(\varepsilon) \) there exists \( J'(I,\varepsilon) \in \mathbb{N}
\), such that the following holds:
\newline
For \( J \geq J'(I,\varepsilon) \) for which we obtain
\[
\left| \frac{1}{J}\sum ^{J}_{j=1}<u_{j}, u_{j+i}>\right| <
\frac{\varepsilon}{2}, \]  for set of \( i \)'s in the interval \(
\{1,\ldots,I\} \) of density \( 1 - \frac{\varepsilon}{3}\) we have
\[
\left\Vert \frac{1}{J}\sum _{j=1}^{J}u_{j}\right\Vert < \varepsilon.
\]
\end{lemma}
 This is a finitary modification of
Bergelson's lemma in \cite{berg_pet}. Its proof may be found in
 \cite{fish2}, lemma 5.4.

\noindent The following lemma is a simple fact that for a weakly mixing
system \( X \) not only an average of shifts for a function converges
to a constant in \( L^2 \) norm but also  weighted average (weights
are bounded) converges to the same constant.
\begin{lemma}
\label{wm_sub_lemma1} Let \( ( X, \mathbb{B}, \mu, T ) \) be a
weakly mixing system  and \( f \in L^{2}(X) \) with \( \int_X f d\mu =
0 \). Let \( \varepsilon > 0 \). Then there exists \( \mathbb{J}
> 0 \) such that for any \( J
> \mathbb{J} \) we have
\[
 \left\| \frac{1}{J} \sum_{j=1}^J b_j T^{j} f \right\|_{L^2(X)} <
 \varepsilon
\]
for any sequence \(  b = (b_1,b_2, \ldots , b_n, \ldots) \in \{ 0 ,1
\}^\mathbb{N}  \).
\end{lemma}
\begin{proof}
Let \( \varepsilon > 0 \).
\newline
 By one of the properties of weak
mixing, for any \( f \in L^2(X) \) with \( \int_X f d \mu(x) = 0 \)
we have \( \frac{1}{N} \sum_{n=1}^N |\left<T^n f, f\right>| \rightarrow 0 \).
\newline
We denote by \( c_n = c_{(-n)}= |\left<T^n f, f \right>| \) and we have that \(
\frac{1}{N} \sum_{n=1}^N c_n \rightarrow 0 \). Then for any \(
\varepsilon > 0 \) there exists \( \mathbb{J} > 0 \) such that for
any \( J
> \mathbb{J} \) we have
\[
\left\| \frac{1}{J} \sum_{j=1}^J b_j T^{j} f \right\|^2 \leq
\frac{1}{J^2} \sum_{j=1,k=1}^J b_j b_k c_{j-k} \leq \frac{1}{J^2}
\sum_{j=1,k=1}^J c_{j-k} \leq \varepsilon.
\]

\hspace{12cm} \qed
\end{proof}


\newpage

\end{document}